\documentclass{amsart}

\usepackage{xcolor}
\usepackage{amscd}
\usepackage{latexsym,amssymb,palatino,color,graphicx,amsmath,amsthm}

\newtheorem{theorem}{Theorem}[section]

\newtheorem{proposition}[theorem]{Proposition}
\newtheorem{corollary}[theorem]{Corollary}

\theoremstyle{definition}
\newtheorem{definition}[theorem]{Definition}

\theoremstyle{remark}

\numberwithin{equation}{section}
\newcommand\style{\mathcal }          %%% calligraphic

\newcommand{\B}{\style{B}}
\newcommand{\M}{\style{M}}

\newcommand\A{{\style A}}
\renewcommand{\H}{\style{H}}
\newcommand{\K}{\style K}

\newcommand\osr{{\style R}}

\newcommand\ost{{\style T}}

\newcommand\omin{\otimes_{\rm min}}
\newcommand\omax{\otimes_{\rm max}}

\newcommand\jay{{\style J}}                                  %%%%% ideal J
\newcommand\cstar{{\rm C}^*}                              %%% C$^*$-algebra generated by
\newcommand\cstare{{\rm C}_{\rm min}^*}              %%% C$^*$-envelope of
\begin{document}

\title[Essential Matrix Ranges 
and the Smith-Ward Problem]{Geometry of Essential Matrix Ranges and
the Smith-Ward Problem for Operator Systems}

\author{Douglas Farenick, Chi-Kwong Li, and Sushil Singla}
\address{Department of Mathematics and Statistics, University of Regina, Regina, Saskatchewan S4S 0A2, Canada}
\address{Department of Mathematics, College of William \& Mary, Williamsburg, Virginia 23185, United States of America}
\address{Department of Mathematics and Statistics, University of Regina, Regina, Saskatchewan S4S 0A2, Canada}

\email{douglas.farenick@uregina.ca}
\email{ckli@math.wm.edu}
\email{sushil@uregina.ca}


 %%%%%%%%%%%%%%%%%%%%% Abstract
\begin{abstract} A $d$-tuple of bounded linear selfadjoint operators acting
on an infinite-dimensional separable Hilbert space is said to have the Smith-Ward property if 
the identity map of the image of the operator system in the Calkin algebra has a completely positive lift. 
In this paper, we focus on noncommutative geometric properties of a finite-dimensional operator system 
with the goal of understanding how geometric information encoded by the essential matrix range of 
a spanning set of linear basis for the operator system implies the Smith-Ward property. Some geometric objects
of special interest in this paper include one form of noncommutative complex Euclidean ball and maximal noncommutative cubes and polydiscs,
as well as some extremal compact matrix convex sets, $K^{\rm min}$ or $K^{\rm max}$, determined by a given compact convex
subset $K$ of $\mathbb R^d$.
\end{abstract}

\maketitle

%%%%%%%%%%%%%%%%%%%%%%%%% Text of the Paper
 \section{Introduction}
 
 A $d$-tuple $\mathfrak a$ of bounded linear selfadjoint operators acting
on a separable Hilbert space of infinite dimension is said to have the Smith-Ward property if 
the essential matrix range $W_{\rm e}(\mathfrak a)$ of 
$\mathfrak a$ agrees with the matrix range $W(\mathfrak a + \mathfrak j)$
of a perturbation of $\mathfrak a$ by a tuple $\mathfrak j$ of compact operators. 
The Smith-Ward problem \cite{smith--ward1980b,smith--ward1980}
is to determine which $d$-tuples of selfadjoint operators exhibit the Smith-Ward property. Every 
selfadjoint operator has the Smith-Ward property, and so the question truly concerns $d$-tuples, for $d\geq 2$. If $d\geq 3$, then there exists
a $d$-tuple of selfadjoint operators that does not have the Smith-Ward property; for $d=2$, it is still an open question 
as to whether all pairs of selfadjoint operators have this property. 

The Smith-Ward problem has a natural formulation in terms of operator-system theory. To explain this formulation, we first settle on some notation.  

If $\osr$ is an operator subsystem of a unital C$^*$-algebra $\A$, namely a linear subspace that contains the multiplicative
identity of $\A$ and is closed under the involution $x\mapsto x^*$, then we let $\cstar(\osr)$ denote the C$^*$-subalgebra
of $\A$ generated by $\osr$. If $\jay\subseteq\A$ is an ideal, then $\pi_\jay$ denotes the canonical $*$-homomorphism $\A\rightarrow \A / \jay$,
and elements $\pi_\jay(x)$ are denoted by $\dot{x}$. If $\phi:\osr\rightarrow\ost$ is any linear map, and $\mathfrak x=(x_1,\dots, x_d)$
is a $d$-tuple of elements from $\osr$, then $\phi(\mathfrak x)$ shall denote the $d$-tuple $(\phi(x_1),\dots, \phi(x_d))$, 
and we use the notation $\dot{\mathfrak x}$ for $\pi_\jay(\mathfrak x)$. Furthermore, 
the operator system determined by $\mathfrak x$ is denoted by $\mathcal O_\mathfrak x$. Thus,
\[
\mathcal O_{\mathfrak x}=\mbox{\rm Span}\{1,x_1,x_1^*,\dots,x_d,x_d^*\},
\]
where $1$ is the multiplicative identity of $\A$ (or the Archimedean order unit of $\osr$, in cases where $\osr$ is not presented explicitly
as an operator subsystem of a particular C$^*$-algebra). Lastly,
the C$^*$-algebra of bounded linear operator acting on the
separable Hilbert space $\ell^2(\mathbb N)$ is denoted by $\mathbb B$, and $\mathbb K$ denotes the ideal of compact operators. 
 
As explained in Proposition \ref{mr-ess} of this paper, 
the Smith-Ward property for a $d$-tuple $\mathfrak a$ of selfadjoint operators in $\mathbb B$ is the property
that the canonical embedding
$\iota:\mathcal O_{\dot{\mathfrak a}}\rightarrow\mathbb B/\mathbb K$ has a completely positive lift to $\mathbb B$. That is, 
the $d$-tuple $\mathfrak a$ has the Smith-Ward
property if and only if there exists a unital completely positive
(ucp) linear map $\psi:\mathcal O_{\dot{\mathfrak a}}\rightarrow\mathbb B$ such that $\iota=\pi_{\mathbb K}\circ\psi$.
More generally, we say that an operator subsystem $\osr\subseteq\mathbb B$ has the Smith-Ward property if the
identity map $\iota:\dot{\osr}\rightarrow\mathbb B/\mathbb K$ has a ucp lift to $\mathbb B$, where $\dot{\osr}=\pi_{\mathbb K}(\osr)$. 
Thus, the Smith-Ward property is a property of operators, $d$-tuples of operators, or operator systems of (Hilbert space) operators exhibited by 
features of their images in the Calkin algebra $\mathbb B/\mathbb K$.

The Brown-Douglas-Fillmore theory of C$^*$-algebra extensions 
\cite{arveson1977,brown--douglas--fillmore1977}
provides a powerful tool to analyse the Smith-Ward property.
Recall, for a unital separable C$^*$-algebra $\A$, an \emph{extension of $\mathbb K$ by $\A$}
 is a unital $*$-monomorphism $\sigma:\A\rightarrow\mathbb B / \mathbb K$.
 The set $\mbox{\rm Ext}(\A)$ of equivalence classes $[\sigma]$ of extensions $\sigma$
 of $\mathbb K$ by $\A$ has the structure of an additive semigroup, and is a group if the C$^*$-algebra $\A$ is nuclear.

 \begin{theorem}\label{inv-ext} If an operator subsystem $\osr$ of
 $\mathbb B $ is such that
 the canonical embedding $\iota:\cstar(\dot{\osr})\rightarrow\mathbb B / \mathbb K$ corresponds to an
 invertible element $[\iota]$ in  $\mbox{\rm Ext}(\cstar(\dot{\osr}))$, then $\osr$ has the Smith-Ward property.
 \end{theorem}
 
 \begin{proof} By \cite[p.~353]{arveson1977}, $[\iota]$ is invertible in $\mbox{\rm Ext}(\cstar(\dot{\osr}))$
 if and only if there is a ucp map $\Psi:\cstar(\dot{\osr})\rightarrow\mathbb B$ such that $\iota=\pi_{\mathbb K}\circ \Psi$.
 The restriction $\psi$ of $\Psi$ to $\dot{\osr}$ is a ucp lift of the canonical embedding
 $\dot{\osr}\rightarrow\mathbb B / \mathbb K$; hence, the operator system $\osr$ has the Smith-Ward property.
 \end{proof} 
 
 \begin{corollary}\label{ext is a group}
 If an operator subsystem $\osr$ of
 $\mathbb B $ is such that
 $\mbox{\rm Ext}(\cstar(\dot{\osr}))$ is a group, then $\osr$ has the Smith-Ward property.
 \end{corollary}
 
 Because $\mbox{\rm Ext}(\A)$ is a group if $\A$ is nuclear, the nuclearity of $\cstar(\dot{\osr})$ is a sufficient condition
 for an operator system $\osr\subseteq\mathbb B$ to have the Smith-Ward property. Further, because the nuclearity of unital C$^*$-algebras
 is preserved under unital $*$-homomorphisms, the nuclearity of $\cstar(\osr)$ is also a sufficient condition for $\osr$ to have
 the Smith-Ward property.
  
 \begin{corollary}\label{nuclear} If $\osr$ is an operator subsystem of $\mathbb B$ such that $\cstar(\osr)$ or $\cstar(\dot{\osr})$ is
 nuclear, then $\osr$ has the Smith-Ward property.
 \end{corollary}
 
 In particular, we also note:
 
 \begin{corollary}\label{abelian} Every $d$-tuple of commuting selfadjoint operators, or  selfadjoint operators commuting modulo $\mathbb K$,
 has the Smith-Ward property.
 \end{corollary}

 These results indicate that, if the C$^*$-algebra $\cstar(\dot{\osr})$ has reasonably good features, such as those
 indicated by Theorem \ref{inv-ext} and Corollary \ref{nuclear},
 then $\osr$ will possess the Smith-Ward property. However, to use these results, 
 one needs sufficient information about the C$^*$-algebra generated by the operator system $\osr$. In contrast,
this paper focuses on noncommutative geometric properties of a finite-dimensional operator system $\osr$ or $\dot{\osr}$,
with the goal of understanding how the geometric information encoded by the matrix range of 
a spanning set of linear basis for $\osr$ or $\dot{\osr}$ implies the Smith-Ward property for the operator system.

An operator system of the form $\mathcal O_\mathfrak x$ is completely determined up to unital complete order isomorphism
by the matrix range $W(\mathfrak x)$ of $\mathfrak x$ \cite{davisdon--dor-on--shalit--solel2017}, which is a graded set of $d$-tuples
of matrices determined by the images $\phi(\mathfrak x)$ of $\mathfrak x$ by matrix states $\phi$ on $\mathcal O_\mathfrak x$. 
Matrix ranges are compact noncommutative convex sets, and many of these sets can be viewed as noncommutative analogues of
specific compact convex subsets of $\mathbb C^d$ or $\mathbb R^d$. For example, the closed unit Euclidean ball in $\mathbb C^d$
admits many noncommutative complex balls of interest \cite{davisdon--dor-on--shalit--solel2017,helton--klep--mcculllough--schweighofer2019}, 
which are graded sets sharing the property that the first set in the grading
is the closed unit Euclidean ball in $\mathbb C^d$. 
As a sample result, we show in Theorem \ref{nc oh cmplx ball} that if the essential matrix range of $\mathfrak x$
is the OH noncommutative complex Euclidean ball, as defined in \cite{helton--klep--mcculllough--schweighofer2019}, 
then the $d$-tuple $\mathfrak x$ has the Smith-Ward property.

 Our main results present a variety of compact matrix convex sets $K_{\rm nc}$ that possess \emph{Smith-Ward geometry}, which is the 
 property that a $d$-tuple $\mathfrak x$ of operators necessarily has the Smith-Ward property if the essential matrix range $W_{\rm e}(\mathfrak x)$ 
 of $\mathfrak x$ is $K_{\rm nc}$. The example mentioned earlier asserts the OH noncommutative complex Euclidean ball has Smith-Ward geometry. 
 
 Two cases of special interest are the compact matrix convex sets $K^{\rm min}$ and $K^{\rm max}$ determined by a given compact convex subset $K$
 of $\mathbb C^d$ or $\mathbb R^d$ (e.g., the closed Euclidean ball).
 We show that the case of $K^{\min}$ can be analysed through the use of function systems, and in this regard it is useful to consider operator systems of 
 $d$-tuples $\mathfrak a$ of commuting selfadjoint operators as operator subsystems of their minimal C$^*$-covers $\cstare(\mathcal O_\mathfrak a)$ \cite{hamana1979b}.
 Motivated by an observation of Kavruk \cite{kavruk2014} that links 
 the Smith-Ward Problem, the Connes Embedding Prolem, and the C$^*$-Nuclearity Problem
 in the setting of $3$-dimensional operator systems, we also show directly that certain Hilbert space operators generate C$^*$-nuclear operator systems,
 including proper isometries and, more generally, contractions $x\in\mathbb B$ with spectrum $\sigma(x)\supseteq S^1$ (Theorem \ref{sw-example}).

%%%%%%%%%%%%%%%%%%%%%%%%%%%%%%%%%%%%%%%%%%%%%  
 \section{Matrix Convexity and the Lifting Property}
%%%%%%%%%%%%%%%%%%%%%%%%%%%%%%%%%%%%%%%%%%%%%
 
 %%%%%%%
 \subsection{Matrix ranges and convexity}
 
 A \emph{matrix state} on an operator system $\osr$ is a ucp map $\phi:\osr\rightarrow\M_n(\mathbb C)$, for some $n\in\mathbb N$.
 The \emph{matrix range} of a $d$-tuple $\mathfrak x$ 
 of elements from an operator system $\osr$ is the graded set $W(\mathfrak x)$ defined by
 \[
W(\mathfrak x)= \bigsqcup_{n\in\mathbb N} W_n(\mathfrak x),
\]
where, for $n\in\mathbb N$, 
\[
W_n(\mathfrak x)=\left\{ \phi(\mathfrak x)\,|\,
\phi:\cstar\left(\mathcal O_{\mathfrak x}\right)\rightarrow\M_n(\mathbb C)
\mbox{ is a matrix state}\right\}.
\]

The first set in the grading of the matrix range $W(\mathfrak x)$ is $W_1(\mathfrak x)$, which is most commonly known
as the \emph{joint numerical range} of the $d$-tuple $\mathfrak x$. 

The following theorem of Davidson, Dor-On, Shalit, and Solel \cite{davisdon--dor-on--shalit--solel2017}
explains how the matrix range $W(\mathfrak x)$
determines the operator system $\mathcal O_{\mathfrak x}$ uniquely up to canonical isomorphism. 

\begin{theorem}\label{ddss} 
If $\mathfrak x$ and $\mathfrak y$ are $d$-tuples of elements in operator systems $\osr$ and $\ost$, respectively,
then the following statements are equivalent:
\begin{enumerate}
\item $W(\mathfrak x)=W(\mathfrak y)$;
\item $\mathcal O_{\mathfrak x} \simeq \mathcal O_{\mathfrak y}$, which is to say
there exists a unital linear bijection
$\phi:\mathcal O_{\mathfrak x}\rightarrow\mathcal O_{\mathfrak y}$ for which both $\phi$ and $\phi^{-1}$ are completely positive 
and such that $\phi(\mathfrak x)=\mathfrak y$.
\end{enumerate}
\end{theorem}

Essential matrix ranges are defined for $d$-tuples of operators acting on $\ell^2(\mathbb N)$. We shall say
a matrix state $\phi$ on an operator subsystem $\osr\subseteq\mathbb B$ is an \emph{essential matrix state} on $\osr$
if $\phi(K)=0$ for every compact operator $K\in\cstar(\osr)$. 

The \emph{essential matrix range} of  $\mathfrak x\in\mathbb B^d$ 
is the graded set $W_{\rm e}(\mathfrak x)$ defined by
 \[
W_{\rm e}(\mathfrak x)= \bigsqcup_{n\in\mathbb N} W_{n,{\rm e}}(\mathfrak x),
\]
where, for $n\in\mathbb N$, 
\[
W_{n,{\rm e}}(\mathfrak x)=\left\{ \phi(\mathfrak x)\,|\,
\phi:\cstar\left(\mathcal O_{{\mathfrak x}}\right)\rightarrow\M_n(\mathbb C)
\mbox{ is an essential matrix state}
\right\}.
\]

It is well-known and simple to verify
that $W_{\rm e}(\mathfrak x)$ and $W(\dot{\mathfrak x})$ coincide as graded sets. Moreover, it is established in 
\cite[Theorem 2.2]{li--paulsen--poon2019} that
\begin{equation}\label{e:lpp}
W_{\rm e}(\mathfrak x)=\displaystyle\bigcap_{\mathfrak j\in\mathbb K^d}W (\mathfrak x + \mathfrak j).
\end{equation}

\begin{proposition}\label{mr-ess} The following statements are equivalent for $\mathfrak x\in\mathbb B^d$:
\begin{enumerate}
\item $\mathfrak x$ has the Smith-Ward property---that is, there exists a $d$-tuple $\mathfrak k$ of compact operators
such that $W(\dot{\mathfrak x})=W(\mathfrak x + \mathfrak k)$;
\item there exists a linear map $\tau:\mathcal O_{\mathfrak x}\rightarrow\mathbb K$ such that 
$W(\dot{\mathfrak x})=W(\mathfrak x + \tau(\mathfrak x))$;
\item the identity embedding $\mathcal O_{\dot{\mathfrak x}}\rightarrow\mathbb B / \mathbb K$ has a ucp lift to $\mathbb B$. 
\end{enumerate}
\end{proposition}

\begin{proof} Assume (1), and let $\mathfrak k$ be a 
$d$-tuple of compact operators
such that $W(\dot{\mathfrak x})=W(\mathfrak x + \mathfrak k)$. 
Hence, there is a unital complete order isomorphism
$\phi:\mathcal O_{\dot{\mathfrak x}}\rightarrow\mathcal O_{{\mathfrak x + \mathfrak k}}$, and so the ucp map
$\phi$ is a ucp lift to $\mathcal B$ of the identity map on $\mathcal O_{\dot{\mathfrak x}}$, which proves (3).

Now assume (3). Thus,
the identity embedding $\iota:\mathcal O_{\dot{\mathfrak x}}\rightarrow\mathbb B / \mathbb K$
has a completely positive lifting $\psi$ to $\mathbb B$.
For each $j$, the operator $\psi(\dot{x}_j)\in\mathbb B$ satisfies $\pi_{\mathbb K}\left( \psi(\dot{x}_j)\right)=\dot{x}_j$. Thus,
$x_j-\psi(\dot{x}_j)\in\ker \pi_{\mathbb K}=\mathbb K$, 
which implies $\psi(\dot{x}_j)=x_j+k_j$, for some compact operator $k_j$. Further, the linear map 
$\tau:\mathcal O_\mathfrak x\rightarrow\mathbb B$ defined by $\tau(y)=\psi(\dot{y})-y$, for $y\in\mathcal O_\mathfrak x$,
has range contained in $\mathbb K$. Hence,
$W(\mathfrak x + \tau(\mathfrak x))\subseteq W(\dot{\mathfrak x})$. 
Thus, using the relation (\ref{e:lpp}), we have
\[
W(\mathfrak x + \tau(\mathfrak x))\subseteq W(\dot{\mathfrak x})=\displaystyle\bigcap_{\mathfrak j\in\mathbb K^d}W (\mathfrak x + \mathfrak j)
\subseteq W(\mathfrak x + \tau(\mathfrak x)),
\]
completing the proof of (2). Statement (1) follows immediately from (2).
\end{proof}

The linearity relationship in (2) of Theorem \ref{mr-ess} above was noted earlier in \cite{robertson--smith1989}.

A \emph{matrix convex set in $\mathbb C^d$} is a graded set $K_{\rm nc}$
defined by
 \[
K_{\rm nc}= \bigsqcup_{n\in\mathbb N} K_n,
\]
where, for $n\in\mathbb N$, $K_n\subseteq \M_n(\mathbb C)^d$, and the following closure property is satisfied:
if, for each $m\in\mathbb N$, whenever $\gamma_j:\mathbb C^n\rightarrow\mathbb C^{n_j}$, for $j=1,\dots m$, 
are linear transformations such that
\begin{equation}\label{e:m-coeff}
\displaystyle\sum_{j=1}^m\gamma_j^*\gamma_j=1_n\in\M_n(\mathbb C),
\end{equation}
then, for any
$\mathfrak a_j\in K_{n_j}$, with $j=1,\dots, m$, we have
\begin{equation}\label{e:m-comb}
\displaystyle\sum_{j=1}^m\gamma_j^*\mathfrak a_j\gamma_j \in K.
\end{equation}
In the sum above, 
the individual summands are defined by
\[
\gamma^*\mathfrak a \gamma=(\gamma^*a_1\gamma, \dots,\gamma^*a_d\gamma),
\]
for every $d$-tuple $\mathfrak a$ of $k\times k$
complex matrices and linear transformation $\gamma:\mathbb C^n\rightarrow\mathbb C^k$.
Sums of the form (\ref{e:m-comb}) in which (\ref{e:m-coeff}) holds are called \emph{matrix-convex combinations} of the 
elements $\mathfrak a_j$ via \emph{matrix-convex coefficients} $\gamma_j$.

If every $K_n$ is compact in $\M_n(\mathbb C)^d$, then the graded set $K_{\rm nc}$ is said to be a \emph{compact matrix convex set}.
A typical example of a compact matrix convex set in $\mathbb C^d$ is the 
matrix range $W(\mathfrak x)$ of a $d$-tuple $\mathfrak x$ of elements from an operator system.

\begin{definition}[Smith-Ward Geometry]\label{swg} 
A compact matrix convex set $K_{\rm nc}$ in $\mathbb C^d$ has \emph{Smith-Ward geometry} if
$\mathfrak x$ has the Smith-Ward property, for every $d$-tuple of $\mathfrak x$ of operators for which 
$W_{\rm e}(\mathfrak x)=K_{\rm nc}$.
\end{definition} 

One of the main contributions of this paper is to identify some important cases of compact matrix convex sets with Smith-Ward geometry.

Each set $K_n$ in the graded matrix convex set $K_{\rm nc}$ is convex, 
and there may be many matrix convex sets $J_{\rm nc}$ in $\mathbb C^d$
that have $J_1=K_1$. This leads to the notion of a  matrix convex set determined by a level-1 (in the grading) set.

\begin{definition}If $K\subseteq \mathbb C^d$ is a compact convex set, then a \emph{noncommutative realisation of $K$}
is a compact matrix convex set $K_{\rm nc}= \bigsqcup_{n\in\mathbb N} K_n$ such that $K_1=K$.
\end{definition}

As shown in \cite{passer--shalit--solel2018},
among all noncommutative realisations of a given compact convex set $K\subseteq\mathbb C^d$, there are two
extremal ones, $K^{\rm min}$ and $K^{\rm max}$, with the property
\[
K^{\rm min}\subseteq K_{\rm nc} \subseteq K^{\rm max},
\]
as graded sets,
for every noncommutative realisation $K_{\rm nc}$ of $K$. 

The sets $K^{\rm min}$ and $K^{\rm max}$ are defined as follows:
\begin{enumerate}
\item $K_n^{\rm max}$ consists of all $d$-tuples $\mathfrak a=(a_1,\dots,a_d)$ 
of $n\times n$ matrices for which $W_1(\mathfrak a)\subseteq K$; and
\item $K_n^{\rm min}$ consists of all $d$-tuples $\mathfrak a=(a_1,\dots,a_d)$ 
of $n\times n$ matrices for which there exist positive 
$h_1,\dots, h_m\in M_n(\mathbb C)$ and $\lambda_j=(\lambda_{j1},\dots,\lambda_{jd})\in K$, for $j=1,\dots, m$, 
such that 
\[
1_n=\sum_{j=1}^m h_j  \,\mbox{ and }\,
a_\ell = \sum_{j=1}^m \lambda_{j\ell}h_j, \,\mbox{ for each }\ell=1,\dots,d.
\]
\end{enumerate}
The graded sets $K^{\rm min}=\displaystyle\bigsqcup_{n\in\mathbb N} K_n^{\rm min}$ and
$K^{\rm max}=\displaystyle\bigsqcup_{n\in\mathbb N} K_n^{\rm max}$ are called the \emph{minimal} and \emph{maximal}
noncommutative realisations of $K$.

From the defining conditions above, 
the elements of $K^{\rm min}$ are those,
and only those, matrix tuples obtained by matrix-convex combinations of elements in $K$. 
Equivalently, a $d$-tuple $\mathfrak x=(x_1,\dots,x_d)$ of matrices
belongs to $K^{\rm min}$ if and only if $x_1,\dots,x_d$ have a joint dilation to $d$ commuting normal operators with joint spectrum
$\overline{\partial_{\rm ext}K}$, where $\partial_{\rm ext}K$ is the set of extreme points of $K$ \cite{davisdon--dor-on--shalit--solel2017}.

The definition of noncommutative realisation of $K$ applies equally well to compact convex subsets $K$ of $\mathbb R^d$.
In such cases, the elements of any noncommutative realisation $K_{\rm nc}$ of $K$ are $d$-tuples of selfadjoint matrices.
Furthermore, it can happen that $K^{\rm min}=K^{\rm max}$; a necessary and sufficient condition for this to occur is that
$K$ be a simplex \cite[Theorem 4.1]{passer--shalit--solel2018}.

Among the compact matrix convex subsets $K_{\rm nc}$ of interest in this paper are
certain noncommutative cubes, polydiscs, and Euclidean balls. In other words,
we examine certain noncommutative realisations
$K_{\rm nc}$ of
\begin{enumerate}
\item[{(i)}] $K=[-1,1]^d$, the $d$-cube in $\mathbb R^d$
obtained via the Cartesian product of $d$ copies of the interval $[-1,1]$, 
\item[{(ii)}] $K=\overline{\mathbb D}^d$, the convex set in $\mathbb C^d$
obtained via the Cartesian product of $d$ copies of the closed unit disc $\overline{\mathbb D}$  in $\mathbb C$, and
\item[{(iii)}] $K=\{\xi\in\mathbb C^d\,|\,|\xi_1|^2+\cdots+|\xi_d|^2\leq 1\}$, the closed 
Euclidean ball in $\mathbb C^d$.
\end{enumerate}

%%%%%%%%
\subsection{Smith-Ward geometry and dilations}

The Stinespring dilation theorem assert that, if $\phi:\A\rightarrow\B(\H)$ is a ucp map, for some unital C$^*$-algebra $\A$
and Hilbert space $\H$, then there exist a unital $*$-representation $\pi:\A\rightarrow\B(\H_\pi)$ of $\A$ on a Hilbert space $\H_\pi$
and a linear isometry $v:\H\rightarrow\H_\pi$ such that $\phi(a)=v^*\pi(x)v$, for every $a\in\A$. If $\A$ is a full matrix algebra $\M_n(\mathbb C)$, 
then $\A$ has a unique irreducible representation up to unitary equivalence; therefore, in such
cases, every unital $*$-representation of $\A$ has the form $x\mapsto x\otimes 1_\K$, for a Hilbert space $\K$ of dimension
determined by the multiplicity of $\pi$. This observation allows one to establish the following dilation theorem.

\begin{theorem}\label{dilation}
If the matrix range $W(\mathfrak a)$ of a $d$-tuple $\mathfrak a$ of matrices has Smith-Ward geometry, then, for every
$\mathfrak x\in\mathbb B^d$ satisfying $W_{\rm e}(\mathfrak x)=W(\mathfrak a)$, there exist a Hilbert space $\K$ and
$\mathfrak j\in\mathbb K^d$ such that $(a_1\otimes 1_\K,\dots,a_d\otimes 1_\K)$ is a dilation of
$(x_1+j_1,\dots,x_d+j_d)$.
\end{theorem}

\begin{proof} Assume $\mathfrak a$ is a $d$-tuple of $n\times n$ matrices
so that $\mathcal O_\mathfrak a\subseteq\M_n(\mathbb C)$.
By hypothesis, $W(\mathfrak x + \mathfrak j)=W(\mathfrak a)$, for some $\mathfrak j\in\mathbb K^d$.
Therefore, $\mathfrak x+\mathfrak j=\phi(\mathfrak a)$, for some ucp map $\phi:\mathcal O_\mathfrak a\rightarrow \mathcal O_{\mathfrak x + \mathfrak j}$.
Extending $\phi$ to a ucp map $\Phi:\M_n(\mathbb C)\rightarrow\mathbb B$, and by 
using the Stinespring dilation theorem and the representation theory discussed above, 
we deduce there are a Hilbert space $\K$ and linear isometry $v:\mathbb C^n\rightarrow\mathbb C^n\otimes\K$
such that $\phi(y)=v^*(y\otimes 1_\K)v$, for every $y\in\M_n(\mathbb C)$.
\end{proof}

%%%%%%%%
\subsection{The lifting property}

A stronger property than the Smith-Ward property is the lifting property for operator subsystems of the Calkin algebra. 

An operator system $\osr$ has the lifting property if every ucp map $\phi:\osr\rightarrow \A/\K$, where $\K$ is a closed ideal of a unital C$^*$-algebra
$\A$, admits a ucp lift $\psi:\osr\rightarrow\A$ such that $\phi=\pi_\K\circ \psi$.
Thus, if $\osr\subseteq\mathbb B$ is an operator subsystem such that $\dot{\osr}$ has the lifting property, then the identity embedding 
$\dot{\osr}\rightarrow\mathbb B / \mathbb K$, in particular, has a ucp lift to $\mathbb B$. In other words, if $\osr\subseteq\mathbb B$ is an 
operator subsystem such that $\dot{\osr}$ has the lifting property, 
then $\osr$ has the Smith-Ward property. 

The following criterion connects the lifting property to operator system tensor products.

\begin{theorem}\label{kavruk lifting}{\rm (\cite[Theorem 8.6]{kavruk--paulsen--todorov--tomforde2013})} 
A finite-dimensional operator system $\osr$ has the lifting property
if and only if $\osr\omin\mathbb B=\osr\omax\mathbb B$. 
\end{theorem}

Another property of operator systems that is stronger than the lifting property is the OMAX property.
An operator system $\osr$ is an \emph{OMAX operator system} if every 
positive linear map $\phi:\osr\rightarrow\ost$, for every operator system $\ost$,
is completely positive. 

\begin{theorem}{\rm (\cite[Lemma 9.10]{kavruk2014})} Every OMAX operator system has the lifting property.
\end{theorem}

The most well-known example of an OMAX operator system is a unital abelian C$^*$-algebra, 
but there are many examples of 
OMAX operator
systems that generate nonabelian C$^*$-algebras.
In particular: any $3$-dimensional
operator system of $2\times 2$ complex matrices forms an OMAX operator 
system (Theorem \ref{quadratic operators}), along with all finite direct sums
of OMAX operator systems \cite{li--poon2020}; the operator system of all complex
$n\times n$ Toeplitz matrices is an OMAX operator
system \cite{farenick2021}; and $\osr\omax\ost$ is an OMAX operator system, if both $\osr$ and $\ost$
are OMAX operator systems \cite{kavruk--paulsen--todorov--tomforde2011}. Properties and additional examples of OMAX operator systems of matrices
are presented in \cite{li--poon2020}, yielding explicit examples of compact matrix convex sets with Smith-Ward geometry to which Theorem \ref{dilation}
applies.

The following theorem is our first, albeit abstract, example of how the geometry of the essential matrix range leads to the Smith-Ward property. Subsequent examples
will be more concrete in that they stem from noncommutative realisations of some particular compact convex sets.

\begin{theorem}\label{lp geom}
If $\mathfrak x$ is a $d$-tuple of elements from an operator system such that $\mathcal O_\mathfrak x$ has the lifting property, then
the matrix range $W(\mathfrak x)$ has Smith-Ward geometry.
\end{theorem}

\begin{proof} If $\mathfrak a\in\mathbb B_{\rm sa}^d$ is such that $W_{\rm e}(\mathfrak a)=W(\mathfrak x)$, then 
the operator systems $\mathcal O_{\dot{\mathfrak a}}$ and $\mathcal O_\mathfrak x$ are unitally completely order isomorphic, by Theorem \ref{ddss}.
Because the lifting property is invariant under unital complete order isomorphisms, $\mathcal O_{\dot{\mathfrak a}}$ has the lifting property, and so 
$\mathfrak a$ has the Smith-Ward property. By definition, this implies the compact matrix convex set $W(\mathfrak x)$ has Smith-Ward geometry.
\end{proof}

The proof of Theorem \ref{lp geom} makes use of the straigthforward fact that if $\mathcal O_\mathfrak a \simeq \mathcal O_\mathfrak b$, 
as defined in Theorem \ref{ddss}, then 
$\mathcal O_\mathfrak a$ has the lifting property if and only if $\mathcal O_\mathfrak b$ has the lifting property. 
The challenge in addressing the Smith-Ward problem
is that the lifting property is stronger than the property that the identity embedding $\ost\rightarrow\mathbb B /\mathbb K$ 
of an operator subsystem $\ost$ of the Calkin
algebra $\mathbb B /\mathbb K$ has a ucp lift to $\mathbb B$. Thus, 
if one has $d$-tuples $\mathfrak a$ and $\mathfrak b$ of operators such that $\mathcal O_{\dot{\mathfrak a}} \simeq \mathcal O_{\dot{\mathfrak b}}$,
then Smith-Ward property of $\mathcal O_a$ may not necessarily transfer to $\mathcal O_b$ through the isomorphism.

%%%%%%%%%%%%%%%%%%%%%%
\section{Noncommutative Squares, Cubes, Polydiscs, and Euclidean Balls}

In this section we show that certain realisations of some important compact convex sets have the Smith-Ward property.

%%%%
 \subsection{Simplicial geometry}
 
 If $K\subseteq\mathbb R^d$ is a simplex, then there is a unique matrix convex set $K_{\rm nc}$ with $K_1=K$, as in this case we
 have $K^{\rm min}=K^{\rm max}$ \cite[Theorem 4.1]{passer--shalit--solel2018}
 
 \begin{proposition}\label{nuc tuples} If $K\subseteq \mathbb R^d$ is a simplex, then 
 the unique matrix convex set $K_{\rm nc}$ for which $K_1=K$ has Smith-Ward geometry.
\end{proposition}

\begin{proof} Because $K$ is a simplex, $K_{\rm nc}=K^{\rm max}=K^{\rm min}$.
Suppose now that $\mathfrak a\in\mathbb B_{\rm sa}^d$ is such that $W_{\rm e}(\mathfrak a)=K^{\rm max}$.
Then, by \cite[Proposition 2.1(2)]{passer--shalit--solel2018}, 
every positive linear map $\phi:\mathcal O_{\dot{\mathfrak a}}\rightarrow \B(\H)$ is a ucp map. 
Hence, $\mathcal O_{\dot{\mathfrak a}}$ is an OMAX operator system and has, therefore, the lifting property.
In other words, the operator $d$-tuple $\mathfrak a$ has the Smith-Ward property.
\end{proof}

 %%%%
 \subsection{Maximal noncommutative squares and polydiscs}
 
 \begin{theorem}\label{nc cube and polydisc} For every $d\in\mathbb N$, there exist $d$-tuples $\mathfrak u$ and $\mathfrak w$ of symmetries and unitaries, respectively, 
such that
\[
\left([-1,1]^d\right)^{\rm max} = W(\mathfrak u)=W_{\rm e}(\mathfrak u)\,\mbox{ and }\, \left(\overline{\mathbb D}^d\right)^{\rm max}= W(\mathfrak w)=W_{\rm e}(\mathfrak w).
\]
Moreover, the compact matrix convex sets $\left([-1,1]^d\right)^{\rm max}$
and $\left(\overline{\mathbb D}^d\right)^{\rm max}$ have Smith-Ward geometry.
\end{theorem}

\begin{proof} Let $*_1^d\mathbb Z_2$ denote the free product of $d$-copies of $\mathbb Z_2$ and $\mathbb F_d$ denote the free group on $d$ generators.
In the group C$^*$-algebras $\cstar(*_1^d\mathbb Z_2)$ and $\cstar(\mathbb F_d)$, the $\mathfrak u$ and $\mathfrak w$ denote $d$-tuples comprised of the canonical unitary generators of
$*_1^d\mathbb Z_2$ and $\mathbb F_d$, respectively, and let $\mathcal O_{\mathfrak u}$ and $\mathcal O_{\mathfrak w}$ be the operator systems determined by these canonical unitary
generators.
Thus, by the universal property of free groups, 
for every Hilbert space $\H$ and $d$-tuples $\mathfrak h=(h_1,\dots,h_d)$ of selfadjoint contractions 
and $\mathfrak x=(x_1,\dots,x_d)$ of contractions acting on $\H$, there exist ucp maps
$\phi:\mathcal O_{\mathfrak u}\rightarrow\mathcal O_{\mathfrak h}$ and $\psi:\mathcal O_{\mathfrak w}\rightarrow\mathcal O_{\mathfrak x}$
such that $\phi(u_j)=h_j$ and $\psi(w_j)=x_j$, for $j=1,\dots,d$.
Hence, 
\[
\left([-1,1]^d\right)^{\rm max} = W(\mathfrak u)\,\mbox{ and }\, \left(\overline{\mathbb D}^d\right)^{\rm max}= W(\mathfrak w).
\]

We now prove that $W(\mathfrak u)=W_{\rm e}(\mathfrak u)$ and $W(\mathfrak w)=W_{\rm e}(\mathfrak w)$.
Assume $\cstar(*_1^d\mathbb Z_2)$ and $\cstar(\mathbb F_d)$ are represented faithfully as unital C$^*$-subalgebras of $\mathbb B$. In place of $\mathfrak u$ and $\mathfrak w$,
take $\mathfrak u^{(\infty)}$ and $\mathfrak w^{(\infty)}$, formed by considering a direct sum of countably many copies of each $u_j$ and $w_k$. Thus, 
$\cstar(\mathcal O_{\mathfrak u^{(\infty)}})\cong\cstar(*_1^d\mathbb Z_2)$ and $\cstar(\mathcal O_{\mathfrak w^{(\infty)}})\cong\cstar(\mathbb F_d)$, 
but $\cstar(\mathcal O_{\mathfrak u^{(\infty)}})$ and $\cstar(\mathcal O_{\mathfrak w^{(\infty)}})$ contain no compact operators. Therefore, without loss of generality,
we may assume $\mathfrak u$ and $\mathfrak w$ are represented faithfully as $d$-tuples of free symmetries and free unitaries such that
$\cstar(\mathcal O_{\mathfrak u})\cap\mathbb K=\cstar(\mathcal O_{\mathfrak w})\cap\mathbb K=\{0\}$. Hence, 
$W(\mathfrak u)=W_{\rm e}(\mathfrak u)$ and $W(\mathfrak w)=W_{\rm e}(\mathfrak w)$.

We now show the assertions regarding Smith-Ward geometry. 
The operator systems $\mathcal O_\mathfrak u$ and $\mathcal O_\mathfrak w$ have the lifting property by
\cite[Theorem 4.3]{farenick--kavruk--paulsen--todorov2018}, for the case of $\mathcal O_\mathfrak u$, and 
 \cite[Theorem 6.3]{farenick--paulsen2012}, for the case of $\mathcal O_\mathfrak w$.
 Thus, if $\mathfrak a$ is a $d$-tuple of selfadjoint operators for which $W_{\rm e}(\mathfrak a)$ is 
 $\left([-1,1]^d\right)^{\rm max}$ or $\left(\overline{\mathbb D}^d\right)^{\rm max}$, then 
 $\mathfrak a$ has the Smith-Ward property by Theorem \ref{lp geom}.
\end{proof}

%%%%%%%%%%%
\subsection{Maximal elliptical discs}

\begin{theorem}\label{ng ellipse}
If $K\subseteq\mathbb R^2$ is a compact convex body such that the topological boundary of $K$ is an ellipse, 
then $K^{\rm max}$ has Smith-Ward geometry.
\end{theorem}

\begin{proof}
Suppose that $\mathfrak a=(a_1,a_2)\in\mathbb B_{\rm sa}^2$ satisfies $W_{\rm e}(\mathfrak a)=K^{\rm max}$, and let $x=a_1+ia_2$.
Thus, in considering $K$ as a nondegenerate elliptical disc in $\mathbb C$, we have that the essential numerical range of $x$
is $K$. As an elliptical disc, the set $K$ also determines
a $2\times 2$ complex matrix $y$, unique up to unitary equivalence, such that $W_1(y)=K$. Moreover, because 
every matrix is unitarily similar to a matrix with constant diagonal entries, we may assume $y$ has the form
\[
y=\left[\begin{array}{cc}
		\alpha&\beta\\\gamma&\alpha
	\end{array}\right],
\]
for some $\alpha, \beta,\gamma\in\mathbb C$. Hence, the operator system $\mathcal O_{\mathfrak y}$,
where $\mathfrak y=(\Re y, \Im y)$, is the operator system of $2\times 2$ complex Toeplitz matrices, 
which is an OMAX operator system \cite{farenick2021}. 

The compact matrix convex set $K^{\rm max}$ is precisely the matrix range of $y$ \cite{tso--wu1999}. Thus,
$W(\dot{\mathfrak a})=W(\dot{x})=K^{\rm max}=W(y)$ implies $\mathcal O_{\dot{\mathfrak a}}\simeq\mathcal O_{\mathfrak y}$.
Hence, $\mathcal O_{\dot{\mathfrak a}}$ is an OMAX operator system, and therefore has the lifting property, 
proving $K^{\rm max}$ has Smith-Ward geometry.
\end{proof}

%%%%%
\subsection{Noncommutative complex Euclidean balls} 
The \emph{OH noncommutaive complex Euclidean $d$-ball} is the graded set $\mbox{\rm OH-Ball}_{\mathbb C}(d)$ of all $d$-tuples 
$\mathfrak c=(c_1,\dots,c_d)$ of complex
matrices for which $\displaystyle\sum_{j=1}^d c_j^*c_j\leq 1$ \cite{helton--klep--mcculllough--schweighofer2019}.

We shall draw substantially on the work of Courtney \cite{courtney2021} to prove the following theorem.
  
\begin{theorem}\label{nc oh cmplx ball} The OH
noncommutative complex Euclidean $d$-ball 
has Smith-Ward geometry.
\end{theorem}

\begin{proof} Consider the universal C$^*$-algebra $\mathcal P_d^{\mathbb C}$ defined by
\[
\mathcal P_d^{\mathbb C}=\cstar \langle z_1,\dots,z_d \,|\,\displaystyle\sum_{j=1}^d z_j^*z_j=1 \rangle .
\]
Let $\mathcal{OP}_d^{\mathbb C}$ denote the operator subsystem of $\mathcal P_d^{\mathbb C}$ determined by the $d$-tuple $\mathfrak z$
of generators $z_1,\dots,z_d$ of $\mathcal P_d^{\mathbb C}$. We claim that $\mbox{\rm OH-Ball}_{\mathbb C}(d)=W(\mathfrak z)$.

If $\mathfrak c\in W(\mathfrak z)$, then there are $n\in\mathbb N$ and a ucp map 
$\phi:\mathcal P_d^{\mathbb C}\rightarrow\M_n(\mathbb C)$ with
$\phi(z_j)=c_j$, for each $j=1,\dots, d$. By the Kadison-Schwarz inequality,
\[
c_j^*c_j = \phi(z_j)^*\phi(z_j) \leq \phi(z_j^*z_j) \Longrightarrow \sum_{j=1}^d c_j^*c_j \leq \phi\left(\sum_{j=1}^d z_j^*z_j\right) = 1,
\]
and so $W_1(\mathfrak c)\subseteq \mbox{\rm OH-Ball}_{\mathbb C}(d)$. Hence, $\mathfrak c\in \mbox{\rm OH-Ball}_{\mathbb C}(d)$.

Conversely, assume $\mathfrak c$ is a $d$-tuple of $n\times n$ matrices
such $\mathfrak c\in \mbox{\rm OH-Ball}_{\mathbb C}(d)$. As is done in the proof of \cite[Proposition 4.1]{courtney2021},
let 
\[
\tilde c_1 = \left[\begin{array}{cc} c_1 & 0 \\ g& 0 \end{array}\right], \,
\tilde c_2 = \left[\begin{array}{cc} c_2 & 0 \\ 0& 1 \end{array}\right], \,\mbox{ and }\,
\tilde c_j = \left[\begin{array}{cc} c_j & 0 \\ 0& 0 \end{array}\right] \,\mbox{ for }j\geq3,
\]
where $g=(1-\displaystyle\sum_{j=1}^d c_j^*c_j)^{1/2}$. 
Because $\displaystyle\sum_{j=1}^d \tilde c_j^*\tilde c_j=1$,
there is a unital $*$-homomorphism $\pi:\mathcal P_d^{\mathbb C}\rightarrow\M_{2n}(\mathbb C)$ such that $\pi(\mathfrak z)=\tilde{\mathfrak c}$.
Composing $\pi$ with the compression to the $n\times n$ upper lefthand corner of $\M_{2n}(\mathbb C)$ leads to a ucp
map $\phi:\mathcal P_d^{\mathbb C}\rightarrow\M_{n}(\mathbb C)$ such that $\phi(\mathfrak z)= \mathfrak c$. Thus, $\mathfrak c\in W(\mathfrak z)$
which completes the proof of $\mbox{\rm OH-Ball}_{\mathbb C}(d)=W(\mathfrak z)$.

Because the C$^*$-algebra $\mathcal P_d^{\mathbb C}=\cstar(\mathcal O_\mathfrak z)$ has the lifting property \cite[Proposition 4.1]{courtney2021},
the matrix range of $\mathfrak z$
has the Smith-Ward property. 
\end{proof}
 
%%%%%%%
\subsection{Minimal noncommutative squares and discs}

The square $\displaystyle\square$ and disc $\overline{\mathbb D}$ 
are the compact convex subsets of $\mathbb C$ given by
\[
\begin{array}{rcl}
\displaystyle\square &=& [-1,1]\times [-1,1]\subset\mathbb C, \\ 
\overline{\mathbb D} &=& \{\zeta\in\mathbb C\,|\,|\zeta|\leq 1\}.
\end{array}
\]

If $u$ and $v$ denote free symmetries---that is, the canonical unitary generators of $\cstar(\mathbb Z_2 * \mathbb Z_2)$---and if
$w$ is a unitary operator such that $w^4=1$ and the spectrum $\sigma(w)$ coincides with the
fourth roots of unity, then the elements $u+iv$ and $(1+i)w$ each have numerical range equal to the square $\displaystyle\square$ and
\[
\begin{array}{rcl}
\displaystyle\square{}^{\rm min} &=& W_{\rm nc}\left( w+iw\right), \mbox{ and} \\
\displaystyle\square{}^{\rm max} &=& W_{\rm nc}\left( u+iv\right).
\end{array}
\]

The minimal and maximal matrix convex sets determined by the closed unit $\overline{\mathbb D}$ are the matrix ranges of the
bilateral shift operator $b$ acting on $\ell^2(\mathbb Z)$ and the nilpotent operator $s=\left[\begin{array}{cc} 0&2 \\ 0&0 \end{array}\right]$
acting on $\mathbb C^2$ \cite{ando1973}.
That is,
\[
\begin{array}{rcl}
\overline{\mathbb D}{}^{\rm min} &=& W_{\rm nc}\left( b\right), \mbox{ and} \\
\overline{\mathbb D}{}^{\rm max} &=& W_{\rm nc}\left( s\right).
\end{array}
\]

An interesting relationship holds
between the minimal noncommutative squares and the maximal noncommutative discs \cite[Theorem 2.5]{choi--li2000}:

\begin{theorem}[Choi-Li]\label{choi-li them} The following statements are equivalent for a matrix $y\in\M_n(\mathbb C)$:
\begin{enumerate}
\item $y\in \displaystyle\square{}^{\rm min}$;
\item $\frac{1}{1+i}\left[\begin{array}{cc} 0 & y^*+y \\ i(y^*-y) & 0 \end{array}\right]\in\overline{\mathbb D}{}^{\rm max}$.
\end{enumerate}
\end{theorem}

By the well-known classical inequalities involving numerical radius and norm, we have the following sharp scaled inclusion:
\begin{equation}\label{e:sc2}
\overline{\mathbb D}{}^{\rm max} \subseteq 2\, \displaystyle \overline{\mathbb D}{}^{\rm min}.
\end{equation}
In the case of the square, 
there is also a sharp scaled inclusion \cite[p.~3233]{passer--shalit--solel2018}:
\begin{equation}\label{e:sc}
\displaystyle\square{}^{\rm max} \subseteq \sqrt{2}\, \displaystyle\square{}^{\rm min}.
\end{equation}

Such scaled inclusions explain the extent to which positive linear maps need to be scaled to become completely positive, as 
shown in the following proposition. 

\begin{proposition} The operator system $\mathcal O_{w+iw}$ is not an OMAX operator system; however, if 
$\psi:\mathcal O_{w+iw}\rightarrow\B(\H)$ is a unital positive linear map, then the unital linear map
$\phi:\mathcal O_{w+iw}\rightarrow\B(\H)$ in which $\phi(w+iw)=\frac{1}{\sqrt{2}}\psi(w+iw)$ and 
$\phi(w^*-iw^*)=\frac{1}{\sqrt{2}}\psi(w^*-iw^*)$ is completely positive.
\end{proposition} 

\begin{proof} The fact that the $\mathcal O_{w+iw}$ is not an OMAX operator system is established in 
\cite{li--poon2020}, while the complete positivity of $\phi$
follows from \cite[Proposition 2.1]{passer--shalit--solel2018} using the scaling constant $\sqrt{2}$
arising in (\ref{e:sc}) above.
\end{proof}

We have determined in Theorem \ref{nc cube and polydisc} that $\displaystyle\square{}^{\rm max}$ and
$\overline{\mathbb D}^{\rm max}$ have Smith-Ward geometry.  We now show
$\displaystyle\square{}^{\rm min}$ and $\overline{\mathbb D}^{\rm min}$ also have Smith-Ward geometry.

\begin{proposition}\label{eg}
The minimal noncommutative square $\displaystyle\square{}^{\rm min}$ and minimal noncommutative disc $\overline{\mathbb D}^{\rm min}$
have Smith-Ward geometry.
\end{proposition}

\begin{proof} Suppose $w$ is a unitary operator with spectrum equal to the 4-th roots of unity. 
Let
$x=\displaystyle\bigoplus_1^\infty(w+iw)\in\mathbb B$ and note that every eigenvalue of $x$ is an isolated eigenvalue
of infinite multiplicity. Thus, 
\[
W_{\rm e}(x)=W(x)=W(w+iw)=\displaystyle\square{}^{\rm min}.
\]
Viewing $w$ as a $4\times 4$ unitary diagonal matrix with diagonal entries $\pm 1, \pm i$, we note that 
\[
\mathcal O_{w+iw}=\{(\alpha+\beta_1,\alpha-\beta_1,\alpha+\beta_2,\alpha-\beta_2)\,|\, \alpha,\beta_1,\beta_2\in\mathbb C\},
\]
which is unitarily equivalent to the operator system dual 
$\mathrm{NC}(2)^\delta$ of the operator subsystem $\mbox{\rm NC}(2)$ of $\cstar(\mathbb Z_2 * \mathbb Z_2)$
determined by the canonical unitary generators of $\cstar(\mathbb Z_2 * \mathbb Z_2)$
\cite[Proposition 6.13]{farenick--kavruk--paulsen--todorov2014}. The operator system $\mbox{\rm NC}(2)$ 
has the property that,  for every 
unital C$^*$-algebra $\A$,
$\mbox{\rm NC}(2)\omin\A=\mbox{\rm NC}(2)\omax\A$ \cite[Theorem 7.3]{farenick--kavruk--paulsen--todorov2014}; 
therefore, its operator
system dual has the same property \cite[Theorem 4.1]{kavruk2015}. Hence, 
$\mathcal O_{w+iw}$ has the lifting property, which proves $\displaystyle\square{}^{\rm min}$ has Smith-Ward geometry.

The case of the minimal noncommutative disc $\overline{\mathbb D}^{\rm min}$ is similar. If $b$ is the bilateral shift operator on $\ell^2(\mathbb Z)$,
then $W(b)= \overline{\mathbb D}^{\rm min}$ and
$\mathcal O_b\simeq \mathcal S_1$, which is the operator system generated by the canonical unitary generator of $\cstar(\mathbb Z)$,
and has the property $\mathcal S_1\omin\A=\mathcal S_1\omax\A$, for every 
unital C$^*$-algebra $\A$ \cite[Propostion 4.3]{farenick--kavruk--paulsen--todorov2014}. Thus, $\mathcal S_1$ has the lifting property, implying
$\overline{\mathbb D}^{\rm min}$ has Smith-Ward geometry.
\end{proof}

%%%%%%%%%%%%
\subsection{Minimal noncommutative power-stretched circles}

The previous examples included the closed unit disc, which is the convex hull of $S^1$.
One can embed $S^1$ into $\mathbb C^d$ with what is called a power stretching of
$S^1$ across $d$ complex dimensions: namely, the 
set $\mathbb S_d^1\subseteq\mathbb C^d$ defined by
\[
\mathbb S_d^1=\left\{(\lambda,\lambda^2,\dots,\lambda^d)\in\mathbb C^d\,|\,\lambda\in S^1\right\}.
\]

\begin{theorem} The minimal noncommutative power-stretched ball
$\left( \mbox{\rm Conv}(\mathbb S_d^1)\right)^{\rm min}$ has Smith-Ward geometry
and consists of all $d$-tuples $\mathfrak x=(x_1,\dots,x_d)$ of matrices for which the block-Toeplitz matrix
\[
 T_\mathfrak x = \left[ \begin{array}{ccccc} 1  & x_1^* & x_2^* & \dots & x_d^* \\
                                                              x_1 & 1 & x_1^* & \ddots & \vdots \\
                                                              x_2 & \ddots & \ddots & \ddots &  x_2^* \\
                                                              \vdots & \ddots & \ddots & \ddots & x_1^* \\
                                                              x_d & \dots & x_2 & x_1 & 1  
 \end{array}\right] 
 \]
is a positive matrix.
\end{theorem}

\begin{proof} It is shown in \cite{farenick2025} that $\left( \mbox{\rm Conv}(\mathbb S_d^1)\right)^{\rm min}=W(\mathfrak b)$,
where $\mathfrak b=(b,b^2,\dots,b^d)$ and $b$ is the bilateral shift operator. It is also shown in \cite{farenick2025} 
that a $d$-tuple $\mathfrak x$ of matrices belongs to $W(\mathfrak b)$ if and only if the matrix $T_\mathfrak x$ is positive.

Note that
\[
\mathcal O_\mathfrak b \simeq \mathcal O_{\dot{\mathfrak b}} \simeq C(S^1)_{(d)},
\]
where $C(S^1)_{(d)}$ is the Fej\'er-Riesz operator subsystem of $C(S^1)$ 
consisting of all $f\in C(S^1)$ having Fourier coefficients $\hat f(n)=0$ for all $n\in\mathbb D$ with $|n|> d$.
The operator system dual of $C(S^1)_{(d)}$ is $C(S^1)^{(d)}$, the operator system of complex $d\times d$
Toeplitz matrices \cite{connes--vansuijlekom2021,farenick2021}. Because $C(S^1)^{(d)}$ is an operator subsystem
of a nuclear C$^*$-algebra, namely $\M_d(\mathbb C)$, $C(S^1)^{(d)}$ is an exact operator system 
\cite[Corollary 5.8]{kavruk--paulsen--todorov--tomforde2013}. Therefore, its operator system dual, $C(S^1)_{(d)}$,
has the lifting property \cite[Theorem 6.6]{kavruk2014}. Hence, $\mathcal O_{\dot{\mathfrak b}}$ has the lifting property,
and so, by Theorem \ref{lp geom}, 
$W(\mathfrak b)$ has Smith-Ward geometry.
\end{proof}
 
%%%%%%%
\subsection{Questions}

In light of the examples presented in this section, one is led to ask:
\begin{enumerate}
\item For which compact convex sets $K\subseteq \mathbb R^d$ does $K^{\rm min}$ have Smith-Ward geometry?
\item For which compact convex sets $K\subseteq \mathbb R^d$ does $K^{\rm max}$ have Smith-Ward geometry?
\end{enumerate}

The first of these questions is considered in the next section by reframing it as a question concerning functions systems.

%%%%%%%%%%%%%%%%%%%%%%%%%%%%%%%%%%%%%%%%
\section{Minimal C$^*$-Covers and Spectral Function Systems}
%%%%%%%%%%%%%%%%%%%%%%%%%%%%%%%%%%%%%%%%%
A \emph{C$^*$-cover} of an operator system $\osr$
is a pair $(\kappa, \A)$ consisting of a unital C$^*$-algebra $\A$
and a unital complete order embedding $\kappa:\osr\rightarrow\A$ such that $\A=\cstar\left(\kappa(\osr) \right)$.
A C$^*$-cover $(\iota,\B)$ is \emph{minimal} if, for every C$^*$-cover $(\kappa, \A)$ of $\osr$, there is a unital $*$-homomorphism
$\pi:\A\rightarrow\B$ such that $\iota=\pi\circ\kappa$. 

The minimal C$^*$-cover exists for every operator system $\osr$, and is unique
up to an isomorphism that fixes $\osr$ \cite{hamana1979b}. The minimal pair is denoted by $\left(\iota_{\rm e},\cstare(\osr)\right)$,
and is known in the literature as the \emph{C$^*$-envelope} of $\osr$.

A key concept for determining minimal C$^*$-covers is that of a boundary representation. If $\osr$ is an operator subsystem of a unital C$^*$-algebra,
then a \emph{boundary representation} of $\osr$ is an irreducible representation $\rho:\cstar(\osr)\rightarrow\B(\H_\rho)$ such that the ucp map
$\phi=\rho_{\vert\osr}:\osr\rightarrow\B(\H_\rho)$ has a unique completely positive extension (namely, $\rho$) to $\cstar(\osr)$. If $\mathcal J$
is the (\v Silov boundary) ideal of $\cstar(\osr)$ consisting of all $a\in\cstar(\osr)$ for which $\rho(a)=0$, 
for every boundary representation $\rho$ of $\osr$, then
the pair $(\pi_\jay, \cstar(\osr) /\jay)$, where $\pi_\jay$ is the canonical quotient homomorphism, 
is a minimal C$^*$-cover of $\osr$ \cite{davidson--kennedy2015}. That is,
$\cstare(\osr)=\cstar(\osr) / \jay$.

Below we deduce the minimal C$^*$-cover of the operator system generated by commuting selfadjoint operators. While this result appears to be
known to experts, we are unaware of any explicit reference.

\begin{theorem}\label{commuting} 
If $\mathfrak a=(a_1,\dots,a_d)$ is a $d$-tuple of
commuting selfadjoint operators, then a character $\rho:\cstar(\mathcal O_{\mathfrak a})\rightarrow\mathbb C$ is a boundary 
representation for $\mathcal O_{\mathfrak a}$ if and only if $\rho(\mathfrak a)$ is an extreme point of the numerical range
$W_1(\mathfrak a)$. Further,
\[
\cstare(\mathcal O_{\mathfrak a}) \cong C\left(\overline{\partial_{\rm ext}W_1(\mathfrak a)}\right).
\]
\end{theorem}

\begin{proof} Assume a character
$\rho:\cstar(\mathcal O_{\mathfrak a})\rightarrow\mathbb C$ is a 
boundary representation for $\mathcal O_{\mathfrak a}$. Thus, the state $\phi=\rho_{\vert\mathcal O_{\mathfrak a}}$ is a pure state
on $\mathcal O_{\mathfrak a}$ \cite[Proposition 2.12]{farenick--tessier2022}, and so $\phi(\mathfrak a)$ is an extreme point of $W_1(\mathfrak a)$
\cite{farenick2000}.

Conversely, suppose $\lambda\in \partial_{\rm ext}W_1(\mathfrak a)$. Thus, there is a 
pure state $\phi$ on $\mathcal O_{\mathfrak a}$ such that $\phi(\mathfrak a)=\lambda$. Let $\mathfrak C_\phi$ be the set of all states on 
$\cstar(\mathcal O_{\mathfrak a})$ extending $\phi$, which is a weak*-closed set in the state space of 
$\cstar(\mathcal O_{\mathfrak a})$. As such, $\mathfrak C_\phi$ is the closed convex hull of its extreme points. Because $\phi$ is pure,
any extreme point $\rho$ of $\mathfrak C_\phi$ is also an extreme point of the state space of $\cstar(\mathcal O_{\mathfrak a})$.
As $\cstar(\mathcal O_{\mathfrak a})$ is abelian, $\rho$ is necessarily multiplicative and is, therefore, uniquely determined by its value on
$\mathfrak a$. Hence, $\mathfrak C_\phi$ is a singleton set, implying $\phi$ is a boundary representation of $\mathcal O_\mathfrak a$,
completing the proof of the first part of the theorem.

Select a countable set $\{\lambda_n\}_n$ of extreme points of $W_1(\mathfrak a)$ such that $\{\lambda_n\}_n$ is dense in 
$\overline{\partial_{\rm ext}W_1(\mathfrak a)}$. Let $\mathfrak d$ be the operator $d$-tuple of diagonal operators formed by taking
a direct sum of the $\lambda_n$. Thus, $\mathfrak d$ has joint spectrum $\overline{\partial_{\rm ext}W_1(\mathfrak a)}$ and
\[
\cstar(\mathcal O_\mathfrak d) \cong C\left( \overline{\partial_{\rm ext}W_1(\mathfrak a)}\right).
\]
If $\rho_n$ denotes the character on $\cstar(\mathcal O_{\mathfrak a})$ induced by the extreme point $\lambda_n$, then
$\pi=\bigoplus_n\rho_n:\cstar(\mathcal O_{\mathfrak a})\rightarrow\cstar(\mathcal O_\mathfrak d)$ is a unital $*$-homomorphism.
Because, by
the noncommutative Choquet Theorem \cite{davidson--kennedy2015}, the norm of any matrix 
$X\in\M_n(\mathcal O_{\mathfrak a})$ is achieved at a boundary representation $\rho$ (that is, $\|X\|=\|\rho^{[n]}(X)\|$), 
the density of the sequence $\{\lambda_n\}_n$ in $\overline{\partial_{\rm ext}W_1(\mathfrak a)}$ implies
the map $\pi$ is a unital complete isometry on $\mathcal O_{\mathfrak a}$. Hence,  the pair
$(\pi, \cstar(\mathcal O_\mathfrak d))$ is a C$^*$-cover
of $\mathcal O_\mathfrak a$ and $\cstare(\mathcal O_\mathfrak a)\cong C\left( \overline{\partial_{\rm ext}W_1(\mathfrak a)}\right)$. 

Finally, let $\jay$ be the \v Silov boundary ideal in $\cstar(\mathcal O_\mathfrak a)$. If $x\in\jay$,
then $\rho_n(x)=0$, for every $n$, and so $x\in\ker\pi$. Conversely, if $x\in\ker\pi$, 
then $\rho_n(x)=0$, for every $n$. 
Because $\cstar(\mathcal O_\mathfrak a)/\jay\cong C\left( \overline{\partial_{\rm ext}W_1(\mathfrak a)}\right)$,
we deduce, by continuity, that $\rho(x)=0$, for every boundary representation $\rho$ of $\mathcal O_\mathfrak a$.
Hence, 
$\cstar(\mathcal O_\mathfrak d)=\cstar(\mathcal O_\mathfrak a) / \jay$,
which proves $\cstar(\mathcal O_\mathfrak d)=\cstare(\mathcal O_\mathfrak a)$.
\end{proof} 

The operator theory above yields the following conclusion for function systems. 

\begin{corollary} If $X\subset \mathbb C$ is a compact set and
\[
\mathcal F_X^{[d]}=\left\{\alpha_0 + \sum_{j=1}^d\beta_j z^j + 
\sum_{\ell=1}^d \gamma_\ell \overline{z}^\ell \,|\,\alpha_0,\beta_j,\gamma_\ell\in\mathbb C\right\}
\subseteq C(X),
\]
then:
\begin{enumerate} 
\item the Choquet boundary of $\mathcal F_X^{[d]}$ is homeomorphic to $\partial_{\rm ext} K_X$, 
where
\[
K_X=\mbox{\rm Conv}\left(\{(\lambda,\lambda^2,\dots,\lambda^d)\,|\,\lambda\in X\}\right);
\]
\item the minimal C$^*$-cover of $\mathcal F_X^{[d]}$ is isomorphic to $C(\overline{\partial_{\rm ext} K_X})$.
\end{enumerate}
\end{corollary}

We now turn to compact subsets $K\subseteq\mathbb R^d$ with the aim of assessing how their geometry relates to the Smith-Ward property.

\begin{definition}\label{defn:fs} If $K\subset\mathbb R^d$ is compact, then the operator subsystem $\mathcal E_K^{[d]}$ of $C(K)$ given by
\[
\mathcal E_K^{[d]}=\mbox{\rm Span}\{1,\varepsilon_1,\dots,\varepsilon_d\},
\]
where each $\varepsilon_j\in C(K)$ is defined by
\[
\varepsilon_j(\zeta_1,\dots,\zeta_d)=\zeta_j,\,\mbox{ for all } (\zeta_1,\dots,\zeta_d)\in K,
\]
is called the \emph{spectral function system determined by $K$}.
\end{definition}

In the case where the compact set $K\subset\mathbb R^d$ is also convex, then the spectral function system determined by $K$
is precisely the function system $A(K)$ is a continuous complex-valued affine functions on $K$.

\begin{theorem}\label{min-lp-swg} If $K\subseteq\mathbb R^d$ is a compact convex set
such that the spectral function system $\mathcal E_K^{[d]}$ determined by $K$ has the lifting property, 
then $K^{\rm min}$ has Smith-Ward geometry.
\end{theorem}
 
 \begin{proof} Suppose $\mathfrak a\in\mathbb B_{\rm sa}^d$ is such that $W_{\rm e}(\mathfrak a)=K^{\rm min}$. By definition
 of $K^{\rm min}$, there exists a
$d$-tuple $\mathfrak d\in\mathbb B_{\rm sa}^d$ of pairwise commuting selfadjoint operators
with joint spectrum $K$. By the Spectral Theorem, $\mathcal O_\mathfrak d\simeq \mathcal E_K^{[d]}$, where each selfadjoint operator
$d_j$ is mapped to $\varepsilon_j$ under the isomorphism. Because $W_{\rm e}(\mathfrak a)=K^{\rm min}$ implies the isomorphism
$\mathcal O_{\dot{\mathfrak a}}\simeq\mathcal O_\mathfrak d\simeq \mathcal E_K^{[d]}$, the operator system
$\mathcal O_{\dot{\mathfrak a}}$ has the lifting property. Hence, $K^{\rm min}$ has Smith-Ward geometry.
\end{proof}

Theorem \ref{min-lp-swg} leads to the following interesting problem: \emph{for which compact subsets $X\subseteq\mathbb R^d$
does the spectral function system $\mathcal E_K$, where $K=\mbox{\rm Conv}\,X$, have the lifting property?}

%%%%%%%%%%%%%%%%%%%%%%%%%%%%%%%%%
\section{C$^*$-Nuclearity, the Connes Embedding Problem, and the Smith-Ward Problem for $d=2$}
%%%%%%%%%%%%%%%%%%%%%%%%%%%%%%%%%

By expressing operators in terms of their real and imaginary parts, the Smith-Ward problem for $d=2$ 
is to determine whether, for each $x\in\mathbb B$, there exists a compact operator $j\in\mathbb K$ such that
$W(x+j)=W_{\rm e}(x)$.  As we have already noted, the Smith-Ward problem has an affirmative answer if it is known that every $3$-dimensional
operator subsystem of the Calkin algebra has the lifting property. Furthermore, as explained in \cite[Theorem 11.5]{kavruk2014}, this problem is
equivalent to the problem of whether every $3$-dimensional operator system has the lifting property. Thus, we are lead to ask: if $\osr$ is a $3$-dimensional
operator system, then is it true that $\osr\omin\mathbb B=\osr\omax\mathbb B$?

One class of operator systems $\osr$ for which $\osr\omin\mathbb B=\osr\omax\mathbb B$ will hold are the \emph{C$^*$-nuclear} operator systems, which are
precisely those operator systems $\osr$ for which $\osr\omin\A=\osr\omax\A$, for every unital C$^*$-algebra $\A$. Therefore, if it were known that every $3$-dimensional 
operator system were C$^*$-nuclear, then we would know that the Smith-Ward problem for $d=2$ is solved affirmatively for all Hilbert space operators.

Kavruk \cite{kavruk2014} observed an intriguing connection between these questions and the Connes Embedding Problem \cite{goldbring2022}: if one could prove that
there exists a non-C$^*$-nuclear $3$-dimensional operator system, and if one could prove 
that the Smith-Ward problem for $d=2$ is solved affirmatively for all operators,
then it would follow that the Connes Embedding Problem necessarily is solved in the negative, thereby giving a purely operator-theoretic proof to the negative solution \cite{nonCEP}
of the Connes Embedding Problem.  Therefore, there is value to studying C$^*$-nuclearity and the Smith-Ward problem in the setting of $3$-dimensional operator systems.

To this end, as indications of what can occur, and of where to not look for possible negative solutions to the Smith-Ward problem,
we present two examples of $3$-dimensional operator systems 
in Theorems \ref{sw-example} and \ref{quadratic operators} below
that are both C$^*$-nuclear and have the Smith-Ward property.

\begin{theorem}\label{sw-example} If $x\in\mathbb B$ is such that $\|x\|\leq 1$ and $\sigma(x)\supseteq S^1$,
then
\begin{enumerate}
\item $\mathcal O_x$ is C$^*$-nuclear,
\item $\cstare(\mathcal O_x)\cong C(S^1)$, and
\item $\mathcal O_x$ has the Smith-Ward property.
\end{enumerate}
\end{theorem}

\begin{proof} Let $u$ denote the unitary generator of $\cstar(\mathbb Z)$. Thus, 
$\sigma(u)=S^1$ and $\mathcal O_u\simeq \mathcal S_1$, where $\mathcal S_1$ is the operator subsystem of $C(S^1)$ spanned by the complex-valued functions
$z\mapsto 1$, $z\mapsto z$, and $z\mapsto\overline z$ on $S^1$.
Because $\mathcal S_1$ is C$^*$-nuclear \cite{kavruk2014,kavruk2015}, so is $\mathcal O_u$.

The matrix range of $u$ is given by $W(u)=\left(\mbox{\rm Conv}(S^1)\right)^{\rm min}$. Because 
$x\in\mathbb B$ is a contraction, we have $x=\phi(u)$ for some ucp map $\phi:\cstar(\mathbb Z)\rightarrow\mathbb B$,
and so $W(x)\subseteq W(u)$. However, as $\sigma(x)\supseteq S^1$, we have
\[
\left(\mbox{\rm Conv}(S^1)\right)^{\rm min} \subseteq W(x)\subseteq W(u) = \left(\mbox{\rm Conv}(S^1)\right)^{\rm min}.
\]
Thus, $\mathcal O_x\simeq \mathcal O_u$ which implies $\mathcal O_x$ is C$^*$-nuclear and $\cstare(\mathcal O_x)=\cstare(\mathcal O_u)$.
As $\cstare(\mathcal S_1)\cong C(S^1)$, we deduce $\cstare(\mathcal O_x)\cong C(S^1)$.

Turning to the proof of the Smith-Ward property, we assume 
$\cstar(\mathcal O_x)\cap\mathbb K\not=\{0\}$, for otherwise $\cstar(\mathcal O_x)$ admits no essential matrix states and
$W(x)=W_{\rm e}(x)$. Hence, under this assumption, $\cstar(\mathcal O_x)$ admits 
essential matrix states. 

By the Kadison-Schwarz inequality, any state $\phi$ on $\cstar(\mathcal O_x)$ that sends $x$ to a point $\lambda\in S^1$
is a character on $\cstar(\mathcal O_x)$; in particular, every essential state on $\cstar(\mathcal O_x)$ that maps $x$ to the unit circle
is an essential character on $\cstar(\mathcal O_x)$, resulting in a point of the essential spectrum $\sigma_{\rm e}(x)=\sigma(\dot{x})$ of $x$.  
Hence, $S^1\subseteq\sigma(\dot{x})$.

Conversely, let $\sigma_{\rm p}(x)$ denote the point spectrum of $x$. 
If $\lambda_1,\lambda_2\in S^1$ are distinct eigenvalues of $x$, with corresponding eigenvectors $\eta_1$ and $\eta_2$, 
then $x^*\eta_j=\overline{\lambda}_j\eta_j$ and
\[
\lambda_1\overline{\lambda}_2\langle\eta_1,\eta_2\rangle=\langle x\eta_1,x\eta_2\rangle=\langle x^*x\eta_1,\eta_2\rangle=\langle\eta_1,\eta_2\rangle.
\]
Because $\lambda_1\overline{\lambda}_2\not=1$, it must be that $\langle\eta_1,\eta_2\rangle=0$. 
Hence, by the separability of $\ell^2(\mathbb N)$, $x$ has at most
countably many eigenvalues on $S^1$, and so $S^1\setminus(S^1\cap\sigma_{\rm p}(x))$ is dense in $S^1$.
The approximate point spectrum of $x$ consists of all eigenvalues of finite multiplicity together with all elements of the left essential spectrum of $x$.
Hence, 
\[
S^1\setminus(S^1\cap\sigma_{\rm p}(x))\subseteq \sigma_{\rm ap}(x)\setminus \sigma_{\rm p}(x)\subseteq\sigma_{\rm le}(x)\subseteq \sigma(\dot{x}).
\]
As $\sigma(\dot{x})$ is closed, 
the closure of $S^1\setminus(S^1\cap\sigma_{\rm p}(x))$, namely $S^1$, is contained in $\sigma(\dot{x})$.

Thus, $\dot{x}$ satisfies $\|\dot{x}\| \leq 1$ and $\sigma(\dot{x})\supseteq S^1$, which implies
$\mathcal O_{\dot{x}}$ is C$^*$-nuclear. Hence, $\mathcal O_{\dot{x}}$ has the lifting property, 
and so $x$ has the Smith-Ward property.
 \end{proof}

The next example concerns quadratic elements, which are nonscalar elements $x$ in a unital C$^*$-algebra $\A$
for which $\alpha x^2 +\beta x +\gamma 1=0$
for some $\alpha, \beta, \gamma\in\mathbb C$ with $\alpha\not=0$.

\begin{theorem}\label{quadratic operators} If $x$ is a quadratic element of a unital C$^*$-algebra $A$, then 
\begin{enumerate}
\item $\mathcal O_x$ is a C$^*$-nuclear OMAX operator system, 
\item $\cstare(\mathcal O_x)=\mathbb C\times \mathbb C$, if $x$ is normal, and
$\cstare(\mathcal O_x)=\M_2(\mathbb C)$, if $x$ is nonnormal.
\end{enumerate}
Furthermore, if $\A=\mathbb B$, then $\mathcal O_x$ has the Smith-Ward property.
\end{theorem}

\begin{proof}
By the polynomial spectral mapping theorem, the spectrum of a quadratic element $x$ has at most two points, which correspond to the roots
of the degree-2 annihilating polynomial $f\in\mathbb C[t]$ of $x$. 

If $x$ is normal, then there must be two distinct spectral points, for otherwise
$x$ would be linear rather than quadratic. Therefore, we obtain $\cstare(\mathcal O_x)=\mathbb C\times \mathbb C$, by Theorem \ref{commuting}.
As $\mathcal O_x$ is unitally completely order isomorphic to an operator subsystem of $\mathbb C^2$, it must be that $\mathcal O_x\simeq\mathbb C^2$,
which is an abelian C$^*$-algebra and, hence, it is both C$^*$-nuclear and OMAX.

Assume now that $x\in\A$
is non-normal. The numerical range $W_1(x)$ is a compact convex body $K$ such that the topological boundary
of $K$ is an ellipse. By Theorem \ref{ng ellipse} and its proof, $W(x)=K^{\rm max}$ has Smith-Ward geometry, and $\mathcal O_x$ 
is an operator system unitally completely order isomorphic to the operator system of complex $2\times 2$ Toeplitz matrices, which
is both C$^*$-nuclear and OMAX, and has $\M_2(\mathbb C)$ as its minimal C$^*$-cover \cite{farenick2021}.

Turning to the case $\A=\mathbb B$, the element
$\dot{x}$ is a scalar, a linear element, or a quadratic element of the Calkin algebra.
In each case, $\mathcal O_{\dot{x}}$ is a C$^*$-nuclear OMAX operator system. Hence, 
$\mathcal O_{\dot{x}}$ has the lifting property, which implies $\mathcal O_{{x}}$ has the Smith-Ward property.
\end{proof}

%%%%%%%%%%%%%%%%%%
\section{Discussion}

In this paper, we have examined the Smith-Ward property as a property of operator systems or matrix convex sets. However, an alternative
and commonly studied approach is through the study of joint numerical ranges and matrix ranges, where a $d$-tuple $\mathfrak x$ of Hilbert space
operators has the Smith-Ward property if there is a $d$-tuple $\mathfrak j$ of compact operators for which $W_{n,\rm e}(\mathfrak x)=W_n(\mathfrak x+\mathfrak j)$,
for all $n\in\mathbb N$. In this direction, it is known from the work in \cite{li--paulsen--poon2019} that for each $n_0\in\mathbb N$ and $d$-tuple $\mathfrak x$ of operators, there
is a $d$-tuple $\mathfrak j$ of compact operators for which $W_{n,\rm e}(\mathfrak x)=W_n(\mathfrak x+\mathfrak j)$,
for all $n\leq n_0$. In this setting, the still-open Smith-Ward problem for $d=2$ is to prove, or disprove, the statement
that to every $x\in\mathbb B$ there corresponds a $j\in\mathbb K$ such that
$W_{n,\rm e}( x)=W_n(x+ j)$, for every $n\in\mathbb N$. Furthermore, some of the results obtained herein regarding simplexes, elliptical discs, and quadratic operators
admit alternative, and sometimes more elementary, proofs using the theory of numerical or matrix
ranges, as demonstrated, for example, in \cite{li--poon2020}.

%%%%%%%%%%%%%%%%%%
\section*{Note Added July 20, 2026}

Two recent advances on the Smith-Ward problem for $d=2$ have been announced by M.~Scherer (arXiv:2607.04274) and S.~Harris (arXiv:2607.11001).

%%%%%%%%%%%%%%%%%%
\section*{Acknowledgement}

We wish to thank the referee of this paper for a very careful review and several helpful comments. 

This research was supported, in part, by
the NSERC Discovery Grant Program (Farenick), the Simons Foundation Grant 851334 (Li), and the PIMS Postdoctoral Fellowship Program (Singla).
Chi-Kwong Li is an affiliate member of the Institute for Quantum Computing, University of Waterloo.

%%%%%%%%%%%%%%%%%%%%%%% References %%%%%%%
%\bibliographystyle{amsplain}
%\bibliography{doug-refs}

\end{document}